\documentclass[psamsfonts]{amsart}
\usepackage{amssymb,amsfonts}
\usepackage{enumerate}
\usepackage{mathrsfs}
\usepackage{url}
\usepackage{graphicx}
\usepackage{arydshln}
\usepackage{mathtools}

\newtheorem{thm}{Theorem}[section]

\newtheorem{prop}[thm]{Proposition}
\newtheorem{lem}[thm]{Lemma}

\newtheorem{quest}[thm]{Question}

\theoremstyle{definition}
\newtheorem{defn}[thm]{Definition}

\newtheorem{notn}[thm]{Notation}

\theoremstyle{remark}
\newtheorem{rem}[thm]{Remark}

\makeatletter
\let\c@equation\c@thm
\makeatother
\numberwithin{equation}{section}

\title{Intermediate arithmetic operations on ordinal numbers}
\author{Harry Altman}
\date{May 22, 2016}
\subjclass[2010]{03E10}
\keywords{ordinal arithmetic, natural operations, Hessenberg operations,
Jacobsthal multiplication, natural exponentiation}
\begin{document}

\begin{abstract}
There are two well-known ways of doing arithmetic with ordinal numbers: the
``ordinary'' addition, multiplication, and exponentiation, which are defined by
transfinite iteration; and the ``natural'' (or ``Hessenberg'') addition and
multiplication (denoted $\oplus$ and $\otimes$), each satisfying its own set of
algebraic laws.  In 1909, Jacobsthal considered a third, intermediate way of
multiplying ordinals
(denoted $\times$), defined by transfinite iteration of natural addition, as
well as the notion of exponentiation defined by transfinite iteration of his
multiplication, which we denote $\alpha^{\times\beta}$.  (Jacobsthal's
multiplication was later rediscovered by Conway.)  Jacobsthal showed these
operations too obeyed algebraic laws.  In this paper, we pick up where
Jacobsthal left off by
considering the notion of exponentiation obtained by transfinitely iterating
natural multiplication instead; we will denote this $\alpha^{\otimes\beta}$.
We show that $\alpha^{\otimes(\beta\oplus\gamma)} =
(\alpha^{\otimes\beta})\otimes(\alpha^{\otimes\gamma})$ and that
$\alpha^{\otimes(\beta\times\gamma)}=(\alpha^{\otimes\beta})^{\otimes\gamma}$;
note the use of Jacobsthal's multiplication in the latter.  We
also demonstrate the impossibility of defining a ``natural exponentiation''
satisfying reasonable algebraic laws.
\end{abstract}

\maketitle

\section{Introduction}
In this paper, we introduce a new form of exponentiation of ordinal numbers,
which we call \emph{super-Jacobsthal exponentiation}, and study its properties.
We show it satisfies two analogues of the usual laws of exponentiation.  These
laws relate super-Jacobsthal exponentiation to other previously studied
operations on the ordinal numbers: natural addition, natural multiplication, and
Jacobsthal's multiplication.  We also show that there is no ``natural
exponentiation'' analogous to natural addition and natural multiplication.

There are two well-known ways of doing arithmetic with ordinal
numbers.  Firstly, there are the ``ordinary'' addition, multiplication, and
exponentiation.  These are defined by starting with the successor operation $S$
and transfinitely iterating; $\alpha+\beta$ is defined by applying to $\alpha$
the successor operation $\beta$-many times; $\alpha\beta$ is $\alpha$ added to
itself $\beta$-many times; and $\alpha^\beta$ is $\alpha$ multiplied by itself
$\beta$-many times.  These also have order-theoretic definitions.

There are also infinitary versions of ordinary addition and ordinary
multiplication, defined for families of operands with a well-ordered index set;
using these, one can write
\[ \alpha\beta = \sum_{i<\beta} \alpha;\qquad
\alpha^\beta = \prod_{i<\beta} \alpha. \]
These can be defined either recursively or order-theoretically.

The ordinary operations obey some of the usual relations between arithmetic
operations:
\begin{enumerate}
\item Associativity of addition: $\alpha+(\beta+\gamma)=(\alpha+\beta)+\gamma$.
\item Left-distributivity of multiplication over addition:
$\alpha(\beta+\gamma)=\alpha\beta+\alpha\gamma$.
\item Associativity of multiplication:
$\alpha(\beta\gamma)=(\alpha\beta)\gamma$.
\item Exponentiation converts addition to multiplication:
$\alpha^{\beta+\gamma}=\alpha^\beta \alpha^\gamma$.
\item Exponential of a product is iterated exponentiation:
$\alpha^{\beta\gamma}=(\alpha^\beta)^\gamma$.
\end{enumerate}
Note that these operations are not commutative; for instance,
$1+\omega=\omega\ne\omega+1$ and $2\omega=\omega\ne\omega2$.  Note further that
distributivity does not work on the right; for instance,
\[ (1+1)\omega = \omega \ne \omega2 = (1\omega) + (1\omega).\]
The infinitary versions of these operations also satisfy analogous laws, which
we will detail later.

Then there are the ``natural'' addition and multiplication, sometimes known as
the Hessenberg operations \cite[pp. 73--81]{hdf}, which we will denote by
$\alpha\oplus\beta$ and $\alpha\otimes\beta$, respectively.  Natural addition
and multiplication can be described as adding and multiplying ordinals as if
they were ``polynomials in $\omega$''; see the next section for a more formal
definition.  These are the operations with which the ordinal numbers embed into
the surreal numbers \cite{ONAG}.  They also have order-theoretic definitions,
due to Carruth \cite{carruth}; see De Jongh and Parikh \cite{wpo} for more on
this.

The natural operations also have infinitary versions, but they are less
well-behaved; see Section~\ref{infsec}.

Now, the operations in the ordinary family were formed by transfinite
iteration; but we can transfinitely iterate the natural operations as well.
E.~Jacobsthal introduced a new sort of multiplication, which he denoted by
$\alpha\times\beta$, by transfinitely iterating natural addition; we call it
``Jacobsthal multiplication''.  It is in a sense intermediate between ordinary
multiplication and natural multiplication.  In fact, one has the inequality
\[ \alpha\beta \le \alpha\times\beta \le \alpha\otimes\beta \]
for all ordinals $\alpha$ and $\beta$.  Jacobsthal then went on and defined a
new form of exponentiation based on transfinitely iterating Jacobsthal
multiplication.  He denoted it by $\alpha^{\underline{\beta}}$, but we will
denote it by $\alpha^{\times\beta}$.  One may consider infinitary Jacobsthal
multiplication as well, so that
\[ \alpha\times\beta = \bigoplus_{i<\beta} \alpha\quad\mathrm{and}\quad
\alpha^{\times\beta} = \bigtimes_{i<\beta} \alpha. \]

Jacobsthal's operations have been rediscovered several times.  In the 1980s,
Jacobsthal's multiplication was rediscovered by Conway and discussed by Gonshor
\cite{gonshor2} and by Hickman \cite{hickman}; as such it has also been referred
to as ``Conway multiplication'', though this name is used also of other
operations.  Both of Jacobsthal's operations were also later rediscovered by
Abraham and Bonnet \cite{1999}.

Just as we may transfinitely iterate natural addition, so may we transfinitely
iterate natural multiplication.  We call the resulting operation
\emph{super-Jacobsthal exponentiation}, and denote it $\alpha^{\otimes\beta}$.
Another way of stating this, again, is that
\[ \alpha^{\otimes\beta} = \bigotimes_{i<\beta} \alpha. \]
This type of exponentiation was previously considered briefly by De Jongh and
Parikh \cite{wpo}, but has otherwise been mostly unexplored.

There are quite a few different notions of addition, multiplication, and
exponentiation being considered here, so we will summarize them with a table to
help clarify the relations between them; see Table~\ref{thetable}.

\begin{table}[htb]
\caption{Each operation is the transfinite iteration of the one above it,
yielding three vertical families of operations, in addition to the diagonal
family of natural operations.  Each operation not on the diagonal,
being a transfinite iteration, is continuous in $\beta$.  In addition, each
operation is pointwise less-than-or-equal-to those on its right; see
Section~\ref{compar}.}
\label{thetable}
\begin{tabular}{rcccc}
\rotatebox{-45}{Natural operations $\rightarrow$} &
\rotatebox{-90}{$S$-based $\rightarrow$} & 
\rotatebox{-90}{$\oplus$-based $\rightarrow$} &
\rotatebox{-90}{$\otimes$-based $\rightarrow$} & \\
\cline{2-2}
\multicolumn{1}{r|}{Successor} &
\multicolumn{1}{c|}{\begin{tabular}{c}
	Successor \\ $S\alpha$ \end{tabular}}
& & & \\
\cline{2-3}
\multicolumn{1}{r|}{Addition} &
\multicolumn{1}{c|}{\begin{tabular}{c}
	Ordinary \\ $\alpha+\beta$ \end{tabular}} &
\multicolumn{1}{c|}{\begin{tabular}{c}
	Natural \\ $\alpha\oplus\beta$ \end{tabular}} & & \\
\cline{2-4}
\multicolumn{1}{r|}{Multiplication} &
\multicolumn{1}{c|}{\begin{tabular}{c}
	Ordinary \\ $\alpha\beta$ \end{tabular}} &
\multicolumn{1}{c|}{\begin{tabular}{c}
	Jacobsthal \\ $\alpha\times\beta$ \end{tabular}} &
\multicolumn{1}{c|}{\begin{tabular}{c}
	Natural \\ $\alpha\otimes\beta$ \end{tabular}} & \\
\cline{2-5}
\multicolumn{1}{r|}{Exponentiation} &
\multicolumn{1}{c|}{\begin{tabular}{c}
	Ordinary \\ $\alpha^\beta$ \end{tabular}} &
\multicolumn{1}{c|}{\begin{tabular}{c}
	Jacobsthal \\ $\alpha^{\times\beta}$ \end{tabular}} &
\multicolumn{1}{c|}{\begin{tabular}{c}
	Super-J. \\ $\alpha^{\otimes\beta}$ \end{tabular}} &
\multicolumn{1}{c|}{\begin{tabular}{c}
	None \\ --- \end{tabular}} \\
\cline{2-5}
\end{tabular}
\end{table}

Note that there is no natural exponentiation to continue the ``diagonal'' family
of natural operations.  We will prove this in Section~\ref{noexpsec}.  (A
version of this theorem was also proven independently by Asper\'o and Tsaprounis
around the same time this paper was being written \cite{longreals}.  Their
desiderata for natural exponentiation are slightly different, but the method of
proof is essentially the same.)

One could continue any of these vertical families further, into higher hyper
operations, as discussed in \cite[pp.  66--79]{bachmann}, but we will not
discuss that possibility here for several reasons, among them that higher hyper
operations lack algebraic properties.

Our main interest here is in the algebraic laws sastisfied by these various
operations, analogous to the algebraic laws satisfied by the ordinary operations
discussed earlier.  Such laws are already known for the natural and Jacobsthal
operations; see Section~\ref{oldsec}.  The main result of this paper is that
super-Jacobsthal exponentiation also satisfies such laws; see
Section~\ref{newsec} for the details.

Before we continue discussing these operations and their laws in more detail,
let us conclude this section with Table~\ref{anothertable} and
Table~\ref{tableinf}, which list out all the relevant algebraic laws in a way
that shows the relations between them.  Table~\ref{anothertable} includes the
finitary versions, while Table~\ref{tableinf} has the infinitary versions.

\begin{table}[htb]
\caption{A table of the (finitary) algebraic laws described in this paper.  Each
law has been placed into one of the three vertical families in
Table~\ref{thetable} based on the ``main'' operation involved, i.e., whichever
one is in the bottom-most row in Table~\ref{thetable} -- note that many of these
laws relate operations in different vertical families, and so would go in more
than one column without this choice of convention.  In addition, the operations
$\oplus$ and $\otimes$ are both commutative, but this is not listed here as it
does not fit into any of the patterns displayed here.}
\label{anothertable}
\begin{tabular}{|c|c|c|}
\hline
\textbf{Successor-based} & \textbf{$\oplus$-based} & \textbf{$\otimes$-based} \\
\hline
\scriptsize{$\alpha+(\beta+\gamma)=(\alpha+\beta)+\gamma$} &
$\alpha\oplus(\beta\oplus\gamma)=(\alpha\oplus\beta)\oplus\gamma$ &
Not applicable \\
\hline
$\alpha(\beta+\gamma)=\alpha\beta+\alpha\gamma$ &
\footnotesize{$\alpha\times(\beta\oplus\gamma)=
	(\alpha\times\beta)\oplus(\alpha\times\gamma)$} &
\footnotesize{$\alpha\otimes(\beta\oplus\gamma)=
	(\alpha\otimes\beta)\oplus(\alpha\otimes\gamma)$} \\
\hline
$\alpha(\beta\gamma)=(\alpha\beta)\gamma$ &
$\alpha\times(\beta\times\gamma)=(\alpha\times\beta)\times\gamma$ &
$\alpha\otimes(\beta\otimes\gamma)=(\alpha\otimes\beta)\otimes\gamma$ \\
\hline
$\alpha^{\beta+\gamma}=\alpha^\beta \alpha^\gamma$ &
$\alpha^{\times(\beta+\gamma)} =
	(\alpha^{\times\beta}) \times (\alpha^{\times\gamma})$ &
$\alpha^{\otimes(\beta\oplus\gamma)} =
	(\alpha^{\otimes\beta}) \otimes (\alpha^{\otimes\gamma})$ \\
\hline
$\alpha^{\beta\gamma}=(\alpha^\beta)^\gamma$ &
$\alpha^{\times(\beta\gamma)}=(\alpha^{\times\beta})^{\times\gamma}$ &
$\alpha^{\otimes(\beta\times\gamma)}=(\alpha^{\otimes\beta})^{\otimes\gamma}$ \\
\hline
\end{tabular}
\end{table}

\begin{table}[htb]
\caption{The infinitary analogue of Table~\ref{anothertable}, organized the same
way.  The associativity laws are stated in an abbreviated form here for
simplicity. The four rows here correspond to the first four rows of
Table~\ref{anothertable}; the fifth row has no extension to the infinitary
setting assuming we use only addition, multiplication, and exponentiation.}
\label{tableinf}
\begin{tabular}{|c|c|c|}
\hline
\textbf{Successor-based} & \textbf{$\oplus$-based} & \textbf{$\otimes$-based} \\
\hline
$\sum_i \sum_j \alpha_{i,j} = \sum_{(j,i)} \alpha_{i,j}$ &
Analogue is false &
Not applicable \\
\hline
$\alpha\sum_i \beta_i = \sum_i \alpha\beta_i$ &
$\alpha\times\bigoplus_i \beta_i = \bigoplus_i (\alpha\times\beta_i)$ &
Analogue is false \\
\hline
$\prod_i \prod_j \alpha_{i,j} = \prod_{(j,i)} \alpha_{i,j}$ &
$\bigtimes_i \bigtimes_j \alpha_{i,j} = \bigtimes_{(j,i)} \alpha_{i,j}$ &
Analogue is false \\
\hline
$\alpha^{\sum_i \beta_i} = \prod_i \alpha^{\beta_i}$ &
$\alpha^{\times(\sum_i \beta_i)} = \bigtimes_i \alpha^{\times\beta_i}$ &
$\alpha^{\otimes(\bigoplus_i \beta_i)}=\bigotimes_i \alpha^{\otimes\beta_i}$ \\
\hline
\end{tabular}
\end{table}

The new results of this paper, then, consist of the laws regarding
super-Jacobsthal exponentiation shown in the tables, and the non-existence of
natural exponentiation.

\section{Operations over the ordinals}
\label{oldsec}

Natural addition and natural multiplication have several equivalent
definitions; the simplest definition is in terms of Cantor normal form.  Recall
that each ordinal number $\alpha$ can be written uniquely as
$\omega^{\alpha_0}a_0 + \ldots + \omega^{\alpha_r}a_r$, where
$\alpha_0>\ldots>\alpha_r$ are ordinals and the $a_i$ are positive integers
(note that $r$ may be $0$); this is known as its Cantor normal form.  (We will
also sometimes, when it is helpful, write $\alpha=\omega^{\alpha_0}a_0 + \ldots
+ \omega^{\alpha_r} a_r + a$ where $a$ is a whole number and $\alpha_r>0$ --
that is to say, we will sometimes consider the finite part of $\alpha$
separately from the rest of the Cantor normal form.)  Then natural addition and
multiplication can roughly be described as adding and multiplying Cantor normal
forms as if these were ``polynomials in $\omega$''.  More formally:
\begin{defn}
We define the \emph{natural sum} of two ordinals $\alpha$ and $\beta$, here
denoted $\alpha \oplus \beta$, as follows.  Take ordinals
$\gamma_0 > \ldots > \gamma_r$ and whole numbers $a_0, \ldots, a_r$ and $b_0,
\ldots, b_r$ so that we may write $\alpha = \omega^{\gamma_0}a_0 + \ldots +
\omega^{\gamma_r}a_r$ and $\beta = \omega^{\gamma_0}b_0 + \ldots +
\omega^{\gamma_r}b_r$.  Then
\[ \alpha\oplus\beta = \omega^{\gamma_0}(a_0+b_0) + \ldots + 
	\omega^{\gamma_r}(a_r+b_r). \]
\end{defn}
\begin{defn}
The \emph{natural product} of $\alpha$ and $\beta$, here denoted $\alpha \otimes
\beta$, is defined as follows.
Write $\alpha = \omega^{\alpha_0} a_0 + \ldots +
\omega^{\alpha_r} a_r$ and  write $\beta = \omega^{\beta_0}b_0 + \ldots +
\omega^{\beta_s}b_s$ with $\alpha_0 > \ldots > \alpha_r$ and $\beta_0 > \ldots >
\beta_s$ ordinals and the $a_i$ and $b_i$ positive integers.  Then
\[ \alpha\otimes\beta = \bigoplus_{\substack{0 \le i \le r \\ 0\le j \le s}}
	\omega^{\alpha_i \oplus \beta_j} a_i b_j. \]
\end{defn}

The natural operations also have recursive definitions, due to Conway \cite[pp.
3--14]{ONAG}.  Let us use the following notation:
\begin{notn}
If $T$ is a set of ordinals, $\sup\nolimits' T$ will denote the smallest ordinal
greater than all elements of $T$.  (This is equal to $\sup\{S\alpha:\alpha\in
T\}$; it is also equal to $\sup T$ unless $T$ has a greatest element, in which
case it is $S(\sup T)$.)
\end{notn}
Then these operations may be characterized by:
\begin{thm}[Conway]
We have:
\begin{enumerate}
\item For ordinals $\alpha$ and $\beta$,
\[ \alpha\oplus\beta = \sup\nolimits'(\{\alpha\oplus\beta':\beta'<\beta\}\cup
\{\alpha'\oplus\beta:\alpha'<\alpha\}). \]
\item For ordinals $\alpha$ and $\beta$,
\[ \alpha\otimes\beta = \min\{x:x\oplus(\alpha'\otimes\beta')>
(\alpha\otimes\beta')\oplus(\alpha'\otimes\beta)\mbox{ for all $\alpha'<\alpha$
and $\beta'<\beta$}\}. \]
\end{enumerate}
\end{thm}

As was mentioned earlier, the natural operations also have order-theoretic
interpretations \cite{carruth,wpo}.

The natural operations have some better algebraic properties than the ordinary
operations -- they are commutative, and have appropriate cancellation
properties; as mentioned earlier, these are the operations with which the
ordinals embed in the field of surreal numbers.  We list out explicitly the
algebraic laws analogous to those satisfied by the ordinary operations:

\begin{lem}
The natural operations satisfy:
\begin{enumerate}
\item Associativity of addition:
$\alpha\oplus(\beta\oplus\gamma)=(\alpha\oplus\beta)\oplus\gamma$.
\item Distributivity of multiplication over addition:
$\alpha\otimes(\beta\oplus\gamma)=
	(\alpha\otimes\beta)\oplus(\alpha\otimes\gamma)$.
\item Associativity of multiplication:
$\alpha\otimes(\beta\otimes\gamma)=(\alpha\otimes\beta)\otimes\gamma$.
\end{enumerate}
\end{lem}
As these operations are commutative, $\otimes$ in fact distributes over $\oplus$
on both sides, but this will not be relevant.

The natural operations do not behave as well as the ordinary operations with
regard to continuity; not being defined by transfinite iteration, these
operations are not continuous in either operand, whereas the ordinary operations
are continuous in the right operand.

As was mentioned earlier, there is no natural exponentiation, and we will prove
this in Section~\ref{noexpsec}.

\subsection{Infinitary ordinary and natural operations}
\label{infsec}

One can, by taking limits, define infinitary versions of these operations as
well.  For instance, for the natural sum, one may define:
\begin{defn}
\label{infoplus}
Given an indexed family of ordinals $\alpha_i$ indexed by the ordinals $i<\beta$
for some ordinal $\beta$, one can define the infinitary natural sum
$\bigoplus_{i<\beta} \alpha_i$:
\begin{enumerate}
\item If $\beta=0$, then $\bigoplus_{i<\beta} \alpha_i=0$.
\item If $\beta=S\gamma$, then
\[ \bigoplus_{i<\beta} \alpha_i =
\left(\bigoplus_{i<\gamma} \alpha_i\right) \oplus \alpha_{\gamma}. \]
\item If $\beta$ is a limit ordinal, then
\[ \bigoplus_{i<\beta} \alpha_i =
\lim_{\gamma<\beta} \bigoplus_{\alpha<\gamma} \alpha_i. \]
\end{enumerate}
\end{defn}
The definition for infinitary natural product is analogous; we will not write it
out explicitly.

Some care is warranted with the infinitary operations, though.  For instance, as
the natural operations are not continuous in the right operand,
$1\oplus(1\oplus1\oplus\ldots)$ is not equal to $1\oplus1\oplus\ldots$ (as
$\omega+1\ne\omega)$, and neither is $2\otimes(2\otimes2\otimes\ldots)$ equal to
$2\otimes2\otimes\ldots$ (as $\omega2\ne\omega)$.  Neither does natural
multiplication distribute over infinitary natural addition; for instance,
$2\otimes(1\oplus1\oplus\ldots)$ is not equal to $2\oplus2\oplus\ldots$, as,
again, $\omega2\ne\omega$.

This is in contrast to the ordinary operations, whose infinitary versions do
satisfy laws extending those in the finitary case.  One has:
\begin{enumerate}
\item Generalized associativity of addition:
\[ \sum_{i<\gamma} \sum_{j<\beta_i} \alpha_{i,j} =
\sum_{(j,i)\in\sum_{k<\gamma} \beta_k} \alpha_{i,j}. \]
\item Left-distributivity of multiplication over addition:
\[ \alpha\sum_{i<\gamma}\beta_i = \sum_{i<\gamma}\alpha\beta_i. \]
\item Generalized associativity of multiplication:
\[ \prod_{i<\gamma} \prod_{j<\beta_i} \alpha_{i,j} =
\prod_{(j,i)\in\sum_{k<\gamma} \beta_k} \alpha_{i,j}. \]
\item Exponentiation converts addition to multiplication:
\[ \alpha^{\sum_{i<\gamma} \beta_i} = \prod_{i<\gamma} \alpha^{\beta_i}. \]
\end{enumerate}
Here, $\sum_{k<\gamma} \beta_k$ is the ordinary sum of the $\beta_k$, which is
considered as a disjoint (tagged) union of the $\beta_k$; each element is an
ordered pair $(j,i)$ for some $i<\gamma$ and some $j<\beta_i$, and they are
ordered lexicographically, first by $i$ and then by $j$.  This same convention
will be used later as well.

It should also be pointed out that while the ordinary operations have a
well-known order-theoretic meaning even when infinitary, the same cannot be said
of the natural operations, whose order-theoretic definitions are not so easy to
extend to the infinitary case.  An order-theoretic characterization of the
infinitary natural sum was recently discovered by P.~Lipparini
\cite{lipparini1,lipparini2}, but none remains known for the infinitary natural
product.

\subsection{Jacobsthal's operations}

In 1909, E.~Jacobsthal introduced \cite{jacobsthal} the operation $\times$,
which we refer to as ``Jacobsthal multiplication'', defined by transfinitely
iterating natural addition; $\alpha\times\beta$ means $\alpha$ added to itself
$\beta$-many times, using natural addition.  More formally:

\begin{defn}[Jacobsthal]
We define the operation $\times$ by
\begin{enumerate}
\item For any $\alpha$, $\alpha\times0:=0$.
\item For any $\alpha$ and $\beta$,
$\alpha\times(S\beta):=(\alpha\times\beta)\oplus\alpha$.
\item If $\beta$ is a limit ordinal,
$\alpha\times\beta:=\lim_{\gamma<\beta} (\alpha\times\gamma)$.
\end{enumerate}
\end{defn}

As noted earlier, this can be equivalently described as
\[ \alpha\times\beta = \bigoplus_{i<\beta} \alpha.\]

This multiplication is not commutative; for instance,
$2\times\omega=\omega\ne\omega2=\omega\times2$.  We will discuss other algebraic
laws for it shortly.

Jacobsthal multiplication can be regarded as intermediate between ordinary and
natural multiplication; like natural multiplication, it is related to natural
addition, but like ordinary multiplication, it is based on transfinite
iteration.  See also Section~\ref{compar}.

Jacobsthal then went on to describe a notion of exponentiation obtained by
transfinitely iterating $\times$, which we refer to as ``Jacobsthal
exponentiation''.  More formally:

\begin{defn}[Jacobsthal]
We define $\alpha^{\times\beta}$ by
\begin{enumerate}
\item For any $\alpha$, $\alpha^{\times0}:=1$.
\item For any $\alpha$ and $\beta$, $\alpha^{\times (S\beta)}:=
(\alpha^{\times\beta})\times\alpha$.
\item If $\beta$ is a limit ordinal,
$\alpha^{\times\beta}:=\lim_{\gamma<\beta} (\alpha^{\times\gamma})$.
\end{enumerate}
\end{defn}

Note that we can define infinitary Jacobsthal multiplication as well, analogous
again to Definition~\ref{infoplus} for the infinitary natural sum; we will not
write this out explicitly.  With this definition, one then has, as noted
earlier, \[ \alpha^{\times\beta} = \bigtimes_{i<\beta} \alpha. \]

Jacobsthal then proved \cite{jacobsthal} the algebraic law:

\begin{thm}[Jacobsthal]
\label{jacthm}
For any ordinals $\alpha$, $\beta$, and $\gamma$, one has
\[ \alpha\times(\beta\oplus\gamma) = 
(\alpha\times\beta) \oplus (\alpha\times\gamma).\]
That is to say, $\times$ left-distributes over $\oplus$.
\end{thm}

This distributivity works only on the left and not on the right; for instance,
\[ (1\oplus1)\times\omega = \omega \ne
\omega2 = (1\times\omega) \oplus (1\times\omega).\]
Jacobsthal gave only a computational proof of Theorem~\ref{jacthm}, by computing
the Cantor normal form of both sides and observing their equality.  More
specifically, he proved:

\begin{thm}[Jacobsthal]
\label{jaccomput}
Let $\alpha$ and $\beta$ be ordinals.  Write $\alpha$ in Cantor normal form as
\[ \alpha=\omega^{\alpha_0}a_0 + \ldots + \omega^{\alpha_r}a_r;\]
here $\alpha_0,\ldots,\alpha_r$ is a decreasing (possibly empty) sequence of
ordinals and the $a_i$ are positive integers.
Write $\beta$ in Cantor normal form as
\[ \beta=\omega^{\beta_0}b_0 + \ldots + \omega^{\beta_s}b_s + b;\]
here $\beta_0,\ldots,\beta_s$ is a decreasing (possibly empty) sequences of
nonzero ordinals, the $b_i$ are positive integers, and $b$ is a nonnegative
integer.
Then
\[ \alpha\times\beta =
\omega^{\alpha_0+\beta_0}b_0 + \ldots + \omega^{\alpha_0+\beta_s}b_s
+ \omega^{\alpha_0}(a_0 b) + \ldots + \omega^{\alpha_r} (a_r b).\]

In other words, if $\beta=\beta'+b$ where $\beta'$ is either $0$ or a limit
ordinal and $b$ is finite, then
\[ \alpha\times\beta = \omega^{\alpha_0} \beta' + \alpha\times b.\]
\end{thm}

With this in hand, Theorem~\ref{jacthm} is straightforward, but as an
explanation, it is not very satisfying.  Here, we improve upon Jacobsthal's
proof by presenting an inductive proof:

\begin{proof}[Inductive proof of Theorem~\ref{jacthm}]
We induct on $\beta$ and $\gamma$.  If $\beta=0$ or $\gamma=0$, the statement is
obvious.  If $\gamma$ is a successor, say $\gamma=S\gamma'$, then we have
\begin{multline*}
\alpha\times(\beta\oplus\gamma) = 
\alpha\times(\beta\oplus S\gamma') =
\alpha\times S(\beta\oplus\gamma') =
(\alpha\times(\beta\oplus\gamma'))\oplus\alpha = \\
(\alpha\times\beta)\oplus(\alpha\times\gamma')\oplus\alpha =
(\alpha\times\beta)\oplus(\alpha\times\gamma),
\end{multline*}
as needed.  If $\beta$ is a successor, the proof is similar.

This leaves the case where $\beta$ and $\gamma$ are both limit ordinals.  Note
that in this case, $\beta\oplus\gamma$ is a limit ordinal as well, and that
\[ \beta\oplus\gamma = \sup(\{\beta\oplus\gamma':\gamma'<\gamma\}\cup
\{\beta'\oplus\gamma:\beta'<\beta\}). \]
So
\begin{multline}
\label{eqn21}
\alpha\times(\beta\oplus\gamma) =
\sup\{\alpha\times\delta : \delta<\beta\oplus\gamma \} = \\
\sup(\{\alpha\times(\beta'\oplus\gamma):\beta'<\beta\}\cup
\{\alpha\times(\beta\oplus\gamma'):\gamma'<\gamma\}) = \\
\sup(\{(\alpha\times\beta')\oplus(\alpha\times\gamma):\beta'<\beta\}\cup
\{(\alpha\times\beta)\oplus(\alpha\times\gamma'):\gamma'<\gamma\}).
\end{multline}

Since $\alpha\times\beta$, $\alpha\times\gamma$, and their natural sum are all
limit ordinals as well, we have
\begin{equation}
\label{eqn22}
(\alpha\times\beta) \oplus (\alpha\times\gamma) =
\sup(\{\delta\oplus(\alpha\times\gamma):\delta<\alpha\times\beta\}\cup
\{(\alpha\times\beta)\oplus\varepsilon:\varepsilon<\alpha\times\beta\}).
\end{equation}

So we want to show that these two sets we are taking the suprema of (in the
final expressions in Equations~\eqref{eqn21} and \eqref{eqn22}) are
cofinal, and thus have equal suprema.  The first of these is actually a subset
of the second, so it suffices to check that it is cofinal in it.  So if
$\delta<\alpha\times\beta$, then $\delta\le\alpha\times\beta'$ for some
$\beta'<\beta$, so $\delta\oplus(\alpha\times\gamma)\le
(\alpha\times\beta')\oplus(\alpha\times\gamma)$; similarly with
$\varepsilon<\alpha\times\gamma$.

So our two suprema are equal and
$\alpha\times(\beta\oplus\gamma)=(\alpha\times\beta)\oplus(\alpha\times\gamma)$;
this proves the theorem.
\end{proof}

Once one has Theorem~\ref{jacthm} in hand, it is straightforward to prove by
transfinite induction, as Jacobsthal did, that

\begin{thm}[Jacobsthal]
\label{jaccor}
The following algebraic relations hold:
\begin{enumerate}
\item Jacobsthal multiplication is associative: For any $\alpha$, $\beta$, and
$\gamma$, one has
	\[\alpha\times(\beta\times\gamma)=(\alpha\times\beta)\times\gamma.\]
\item Jacobsthal exponentiation converts ordinary addition to Jacobsthal
multiplication: For any $\alpha$, $\beta$, and $\gamma$, one has
	\[\alpha^{\times(\beta+\gamma)}=
	(\alpha^{\times\beta})\times(\alpha^{\times\gamma}).\]
\item The Jacobsthal exponential of an ordinary product is an iterated
Jacobsthal exponentiation: For any $\alpha$, $\beta$, and $\gamma$, one has
	\[\alpha^{\times(\beta\gamma)}=(\alpha^{\times\beta})^{\times\gamma}.\]
\end{enumerate}
\end{thm}

The same methods easily show infinitary versions of these.
\begin{thm}
\label{jaccorinf}
The following algebraic relations hold:
\begin{enumerate}
\item Jacobsthal multiplication distributes over infinitary natural sum:
\[ \alpha\times\bigoplus_{i<\gamma}\beta_i =
\bigoplus_{i<\gamma}(\alpha\times\beta_i). \]
\item Infinitary Jacobsthal multiplication satisfies ``generalized
associativity'':
\[ \bigtimes_{i<\gamma} \bigtimes_{j<\beta_i} \alpha_{i,j} =
\bigtimes_{(j,i)\in\sum_{k<\gamma} \beta_k} \alpha_{i,j} \]
\item Jacobsthal exponentiation converts infinitary addition to Jacobsthal
multiplication:
\[ \alpha^{\times(\sum_{i<\gamma} \beta_i)} =
\bigtimes_{i<\gamma} \alpha^{\times\beta_i}. \]
\end{enumerate}
\end{thm}

\subsection{Jacobsthal's laws: Discussion}
\label{disc}

We have just given an inductive proof of Theorem~\ref{jacthm}.  However, one
obvious question remains: Is there an order-theoretic proof?  We can ask the
same for Theorems~\ref{jaccor} and \ref{jaccorinf} as well.  Of course, to write
an order-theoretic proof of any of these, one would first need an
order-theoretic interpretation of Jacobsthal multiplication.

As mentioned earlier, however, an order-theoretic characterization of the
infinitary natural sum was recently found by P.\ Lipparini
\cite{lipparini1,lipparini2}, which in particular yields an order-theoretic
characterization of Jacobsthal multiplication.  This charaterization does not
make Theorem~\ref{jacthm} or part (1) of Theorem~\ref{jaccor} obvious, so there
is still work to do there, but an answer may be close at hand.  As for parts (2)
and (3), no order-theoretic intepretation has yet been found for Jacobsthal
exponentiation, or for infinitary Jacobsthal multiplication more generally.

There is an additional mystery to part (1) of Theorem~\ref{jaccor}.  While the
proof is a simple transfinite induction using Theorem~\ref{jacthm}, the
statement itself still looks strange; why should the operation of $\times$ be
associative?  Typically, when we prove that an operation $*$ is associative, we
are not just proving that $a*(b*c)=(a*b)*c$; rather, we usually do it by proving
that $a*(b*c)$ and $(a*b)*c$ are both equal to some object $a*b*c$, and that
indeed $a_1*\ldots*a_r$ makes sense for any finite $r$ -- not just proving that
this makes sense \emph{because} $*$ happens to be associative, so that $a*b*c$
is may be written as a notational shortcut; but that $a*b*c$ makes sense as an
object on its own, and that this relation is \emph{why} $*$ must be associative.
The same question applies, perhaps even more so, to part (2) of
Theorem~\ref{jaccorinf}.  (Note that the generalized associativity laws
satsified by ordinary sum and ordinary product have both been stated in this
relation-between-arities form, because this is the simplest way to do so.)

Consider, for instance, multiplication of cardinal numbers; the simplest way to
show associativity of the binary version is to first define it for any number of
operands.  One would define the product $\kappa\lambda\mu$ to be the cardinality
of the Cartesian product $\kappa\times\lambda\times\mu$, a set of ordered
triples, and then observe that
$\kappa(\lambda\mu)=\kappa\lambda\mu=(\kappa\lambda)\mu$.  Multiplication of
cardinal numbers actually provides an especially clear illustration of this
tendency, if one considers the infinitary version.  Whereas a finitary product
of cardinals, though it may be taken all at once as described, may also be
broken down in terms of iterated binary products, an infinitary product of
cardinals cannot be written as a limit of finitary products in the obvious
fashion; it must be taken all at once.  But with Jacobsthal multiplication --
unlike, say, with ordinary multiplication of ordinals, where the infinitary
product has a clear order-theoretic meaning -- it's not clear what it would mean
to take the product all at once, how one would define it other than as a limit
of iterated binary products.  Even though the infinitary version was stated in
the form of relation between arities, for now those higher arities remain simply
a notational convention.  (Infinitary natural multiplication has a lesser
version of the same problem, of course,  since there is still no known
interpretation of the infinitary natural product other than as a limit; but
there at least finite products make sense taken all at once, without recourse to
iteration.)

So we ask the questions:

\begin{quest}
Can Theorem~\ref{jacthm} be proven by giving an order-theoretic
interpretation to both sides?  Can the same be done for the various parts of
Theorem~\ref{jaccor} and Theorem~\ref{jaccorinf}?
\end{quest}

\begin{quest}
Can the associativity of Jacobsthal multiplication be proven by finding a
natural way of interpreting $\alpha\times\beta\times\gamma$ without first
inserting parentheses?  Can the same be done for the infinitary version, finding
a way of interpreting $\bigtimes_i \alpha_i$ other than as a limit?
\end{quest}

To go in a different direction, rather than restricting surreal operations to
the ordinals, or trying to define a natural exponentiation on the ordinals
analogous to surreal exponentiation, one could also attempt to extend the
ordinary ordinal operations, or these intermediate ones, to the surreal numbers.
This was accomplished for ordinary addition by Conway \cite[Ch.\ 15]{ONAG};
indeed, he extended it to all games, not just numbers.  For ordinary
multiplication, there is a definition of S.\ Norton which was proven by P.\
Keddie \cite{keddie} to work for surreal numbers written in a particular form,
namely, those written with no reversible options; see his paper for more.  It
remains to be seen whether this can be done for Jacobsthal multiplication, or
for any of the exponentiation operations considered here; Keddie \cite{keddie}
gives reasons why this may be difficult for exponentiation.

\section{Super-Jacobsthal exponentiation}
\label{newsec}

Having discussed Jacobsthal's operations, there is still one spot missing from
Table~\ref{thetable}: The transfinite iteration of natural multiplication, or
``super-Jacobsthal exponentiation'', as we call it here.  (Rather, it is the one
spot still missing that actually exists.)  As mentioned earlier,
it was considered briefly by De Jongh and Parikh \cite{wpo}, but has otherwise
remained mostly unexplored.

\begin{defn}
\label{superdef}
We define $\alpha^{\otimes\beta}$ by
\begin{enumerate}
\item For any $\alpha$, $\alpha^{\otimes0}:=1$.
\item For any $\alpha$ and $\beta$, $\alpha^{\otimes (S\beta)}:=
(\alpha^{\otimes\beta})\otimes\alpha$.
\item If $\beta$ is a limit ordinal,
$\alpha^{\otimes\beta}:=\lim_{\gamma<\beta} (\alpha^{\otimes\gamma})$.
\end{enumerate}
\end{defn}

An equivalent way of stating this, as mentioned earlier, is that
\[ \alpha^{\otimes\beta} = \bigotimes_{i<\beta} \alpha. \]

Before we continue, it is worth noting that all the notions of multiplication
and exponentiation considered here are in fact different.  An example is
provided by considering $(\omega+2)(\omega+2)$, or $(\omega+2)^2$, since one has
the equations
\begin{eqnarray*}
(\omega+2)^2 & = & \omega^2 + \omega 2 + 2,\\
(\omega+2)^{\times2} & = & \omega^2 + \omega 2 + 4,\\
(\omega+2)^{\otimes2} & = & \omega^2 + \omega 4 + 4.\\
\end{eqnarray*}
With Definition~\ref{superdef} in hand, we can now state:

\begin{thm}
\label{mainthm}
For any ordinals $\alpha$, $\beta$, and $\gamma$, one has
\[ \alpha^{\otimes(\beta\oplus\gamma)} = 
(\alpha^{\otimes\beta}) \otimes (\alpha^{\otimes\gamma}).\]
That is to say, super-Jacobsthal exponentiation converts natural addition to
natural multiplication.
\end{thm}

Before we prove this theorem, let us make some further notes.  Once it is
proven, it will be straightforward to prove by transfinite induction
that
\begin{thm}
\label{maincor}
For any ordinals $\alpha$, $\beta$, and $\gamma$, one has
\[ \alpha^{\otimes(\beta\times\gamma)} = 
(\alpha^{\otimes\beta})^{\otimes\gamma}.\]
That is to say, the super-Jacobsthal exponential of a Jacobsthal product is an
iterated super-Jacobsthal exponential.

More generally, given ordinals $\alpha$ and $\gamma$ and a family of ordinals
$\beta_i$ indexed by $\gamma$, one has
\[ \alpha^{\otimes(\bigoplus_{i<\gamma} \beta_i)} =
\bigotimes_{i<\gamma} \alpha^{\otimes \beta_i}. \]
\end{thm}

Once this is proven, it will complete Tables~\ref{anothertable} and
\ref{tableinf}.

Note the appearance of Jacobsthal multiplication -- not ordinary or natural
multiplication -- on the left-hand side of the first equation.  This occurs
because Theorem~\ref{maincor} comes from transfinitely iterating
Theorem~\ref{mainthm}, and when one transfinitely iterates natural addition, one
gets Jacobsthal multiplication.

Now we prove Theorem~\ref{mainthm}.  This will require a tiny bit more setup.
First, some notation and two lemmas:

\begin{notn}
For an ordinal $\alpha$ which is either $0$ or a limit ordinal,
$\omega^{-1}\alpha$ will denote the unique ordinal $\beta$ such that
$\alpha=\omega\beta$.
\end{notn}

\begin{notn}
For an ordinal $\alpha>0$, $\deg \alpha$ will denote the largest exponent
appearing in the Cantor normal form of $\alpha$.
\end{notn}

\begin{lem}
\label{fincomput}
Suppose $a>1$ is finite and and let $\beta$ be an ordinal.  Write
$\beta=\beta'+b$, where $\beta'$ is $0$ or a limit ordinal and $b$ is finite.
Then \[ a^{\otimes\beta} = \omega^{\omega^{-1}\beta'} a^b. \]
\end{lem}

\begin{proof}
We induct on $\beta$.
If $\beta=0$, then both sides are equal to $1$.  If $\beta$ is a successor
ordinal, say $\beta=S\gamma$, then by the inductive hypothesis,
\[ a^{\otimes\gamma} = \omega^{\omega^{-1}\gamma'} a^c,\]
where we write $\gamma=\gamma'+c$ analogously to $\beta=\beta'+b$.  As
$\beta=S\gamma$, we have $\beta'=\gamma'$ and $b=c+1$.  Thus
\[ a^{\otimes\beta} = a^{\otimes\gamma} \otimes a =
(\omega^{\omega^{-1}\gamma'} a^c) \otimes a =
\omega^{\omega^{-1}\beta'} a^b.\]

If $\beta$ is a limit ordinal, we have two further cases, depending on whether
or not $\beta$ is of the form $\omega^2 \gamma$ for some ordinal $\gamma$.  If
not, then $\beta$ is of the form $\gamma'+\omega$, where $\gamma'$ is either $0$
or a limit ordinal.  This means that $\beta$ is the limit of $\gamma'+c$ for
finite $c$.  So then by the inductive hypothesis,
\[ a^{\otimes\beta} = \lim_{c<\omega} (\omega^{\omega^{-1}\gamma'}a^c)
= \omega^{S(\omega^{-1}\gamma')} = \omega^{\omega^{-1}\beta'},\]
as required.

If so, then we once again consider $\deg a^{\otimes\beta}$.  Since $\beta$
is of the form $\omega^2 \gamma$, $\beta$ is the limit of all ordinals less than
it of the form $\omega \gamma$, i.e., it is the limit of all limit ordinals less
than it.  And for $\gamma<\beta$ a limit ordinal, by the inductive hypothesis,
$\deg a^{\gamma} = \omega^{-1}\gamma$.  So again applying the fact that the
$\deg$ function is increasing, we have that $\deg a^{\otimes\beta}\ge
\omega^{-1}\beta$, i.e., that $a^{\otimes\beta} \ge
\omega^{\omega^{-1}\beta}$.
(Here we also use the continuity of ``division by $\omega$'', which follows from
the continuity of left-multiplication by $\omega$.)
Conversely, for $\gamma<\beta$ with $\gamma$ a limit ordinal, one has
$\omega^{-1}\gamma<\omega^{-1}\beta$, and so
$a^{\otimes\gamma}<\omega^{\omega^{-1}\beta}$; thus one has
$a^{\otimes\beta}\le\omega^{\omega^{-1}\beta}$.  So we conclude, as needed,
that $a^{\otimes\beta}=\omega^{\omega^{-1}\beta}$.  This proves the lemma.
\end{proof}

\begin{notn}
For ordinals $\alpha$ and $\beta$, $\alpha\ominus\beta$ will denote the smallest
$\gamma$ such that $\beta\oplus\gamma\ge\alpha$.  For convenience, we will also
define
\[f_{\alpha,\beta}(\alpha',\beta') = 
((\alpha\otimes\beta')\oplus(\alpha'\otimes\beta))\ominus(\alpha'\otimes\beta').
\]
\end{notn}

Note that with this definition, we can rewrite Conway's definition of
$\alpha\otimes\beta$ as

\[ \alpha\otimes\beta = \sup\nolimits'\{f_{\alpha,\beta}(\alpha',\beta'):
\alpha'<\alpha,\beta'<\beta\}. \]

\begin{lem}
\label{monoton}
For fixed $\alpha$ and $\beta$, $f_{\alpha,\beta}(\alpha',\beta')$ is increasing
in $\alpha'$ and $\beta'$.
\end{lem}

\begin{proof}
Observe that $f_{\alpha,\beta}(\alpha',\beta')$ is the smallest ordinal greater
than the surreal number $\alpha'\beta+\alpha\beta'-\alpha'\beta'$ (where these
operations are performed in the surreal numbers, and are therefore natural
operations on the corresponding ordinals).  This expression is increasing in
$\alpha'$ and $\beta'$, since it can be written as
$\alpha\beta-(\alpha-\alpha')(\beta-\beta')$.  Therefore so is
$f_{\alpha,\beta}(\alpha',\beta')$, the smallest ordinal greater than it.
\end{proof}

Now, the proof:

\begin{proof}[Proof of Theorem~\ref{mainthm}]
We split this into several cases depending on the value of $\alpha$.  If
$\alpha\in\{0,1\}$ the theorem is obvious.

Now we have the case where $\alpha>1$ is finite; in this case we will use
Lemma~\ref{fincomput} to give a computational proof.  Let us rename $\alpha$ to
$a$ to make it clear that it is finite.  Let $\beta=\beta'+b$ and
$\gamma=\gamma'+c$ where $\beta'$ and $\gamma'$ are limit ordinals or $0$, and
$b$ and $c$ are finite.

So observe first that
\[ \omega^{-1}(\beta'\oplus\gamma') =
\omega^{-1}\beta' \oplus \omega^{-1}\gamma'.\]
This can be seen as, if $\beta'=\omega\beta''$ and $\gamma'=\omega\gamma''$,
then
\[ \omega(\beta''\oplus \gamma'') = \omega\beta'' \oplus \omega\gamma'', \]
which can be seen by comparing Cantor normal forms.  (This can also be seen by
noting that for any ordinal $\delta$, $\omega\delta=\omega\times\delta$, since
if $\varepsilon$ is a limit ordinal then
$\varepsilon\oplus\omega=\varepsilon+\omega$, and by induction this quantity
will always be a limit ordinal.)

Now, $\beta\oplus\gamma$
can be written as $(\beta'\oplus\gamma')+(b+c)$; here, $\beta'\oplus\gamma'$ is
either $0$ or a limit ordinal, and $b+c$ is finite.  Thus,
\begin{multline*}
a^{\otimes(\beta\oplus\gamma)} = 
\omega^{\omega^{-1}(\beta'\oplus\gamma')} a^{b+c} =
\omega^{(\omega^{-1}\beta')\oplus(\omega^{-1}\gamma')} a^b a^c =\\
(\omega^{\omega^{-1}\beta'} a^b) \otimes
	(\omega^{\omega^{-1}\gamma'} a^c) =
a^{\otimes \beta} \otimes a^{\otimes \gamma},
\end{multline*}
as required.

This leaves the case where $\alpha$ is infinite.  In this case we give an
inductive proof, inducting on $\beta$ and $\gamma$.  If $\beta=0$ or $\gamma=0$
the theorem is obvious.  If $\gamma$ is a successor ordinal, say
$\gamma=S\gamma'$, then
\begin{multline*}
\alpha^{\otimes(\beta\oplus\gamma)} = 
\alpha^{\otimes(\beta\oplus S\gamma')} =
\alpha^{\otimes S(\beta\oplus\gamma')} =
\alpha^{\otimes(\beta\oplus\gamma')}\otimes\alpha = \\
\alpha^{\otimes\beta}\otimes\alpha^{\otimes\gamma'}\otimes\alpha =
\alpha^{\otimes\beta}\otimes\alpha^{\otimes\gamma},
\end{multline*}
as needed.  If $\beta$ is a successor, the proof is similar.

This leaves the case where $\beta$ and $\gamma$ are both limit ordinals.  As
before, not only are $\beta$ and $\gamma$ limit ordinals but so is
$\beta\oplus\gamma$.  So
\begin{multline}
\label{eqn23}
\alpha^{\otimes\beta\oplus\gamma} =
\sup\{\alpha^{\otimes\delta}:\delta<\beta\oplus\gamma\} = \\
\sup(\{\alpha^{\otimes(\beta'\oplus\gamma)}:\beta'<\beta\}\cup
\{\alpha^{\otimes(\beta\oplus\gamma')}:\gamma'<\gamma\})
\end{multline}

On the other hand,
\begin{multline}
\label{eqn24}
\alpha^{\otimes\beta}\otimes\alpha^{\otimes\gamma} =
\sup\{f_{\alpha^{\otimes\beta},\alpha^{\otimes\gamma}}(\delta,\varepsilon):
\delta<\alpha^{\otimes\beta},\varepsilon<\alpha^{\otimes\gamma}\} = \\
\sup\{f_{\alpha^{\otimes\beta},\alpha^{\otimes\gamma}}
(\alpha^{\otimes\beta'},\alpha^{\otimes\gamma'}):
\beta'<\beta,\gamma<\gamma'\} =\\
\sup\{((\alpha^{\otimes\beta'}\otimes\alpha^{\otimes\gamma})\oplus
(\alpha^{\otimes\beta}\otimes\alpha^{\otimes\gamma'}))\ominus
(\alpha^{\otimes\beta'}\otimes\alpha^{\otimes\gamma'})
:\beta'<\beta,\gamma'<\gamma\} =\\
\sup\{(\alpha^{\otimes(\beta'\oplus\gamma)}\oplus
\alpha^{\otimes(\beta\oplus\gamma')})\ominus
\alpha^{\otimes(\beta'\oplus\gamma')}
:\beta'<\beta,\gamma'<\gamma\}.
\end{multline}

Note that here we have used not only the inductive hypothesis, but have also
used Lemma~\ref{monoton} and the fact that $\alpha^{\otimes\gamma}$,
$\alpha^{\otimes\beta}$, and their natural product are all limit ordinals.

So now once again we must show that the two sets we are taking the suprema of in
the final expressions of Equations~\eqref{eqn23} and \eqref{eqn24} are cofinal
with each other.  Let us call these sets $S$ and $T$, respectively.

So let us take an element of $S$; say it is
$\alpha^{\otimes(\beta'\oplus\gamma)}$ for $\beta'<\beta$.  We want to show it
is bounded above by some element of $T$.  (If instead it is of the form
$\alpha^{\otimes(\beta\oplus\gamma')}$ for $\gamma'<\gamma$, the proof is
similar.)  But certainly, choosing $\gamma'=0$,
\[ \alpha^{\otimes(\beta'\oplus\gamma)}\oplus\alpha^{\otimes\beta'}
< \alpha^{\otimes(\beta'\oplus\gamma)}\oplus\alpha^{\otimes\beta} \]
and so 
\[ \alpha^{\otimes(\beta'\oplus\gamma)} <
(\alpha^{\otimes(\beta'\oplus\gamma)}\oplus\alpha^{\otimes\beta})
\ominus\alpha^{\otimes\beta'}. \]

Conversely, say we take an element $\delta$ of $T$.  Since we assumed $\alpha$
infinite, and in general we have
$\deg(\alpha\otimes\beta)=(\deg\alpha)\oplus(\deg\beta)$, it follows that the
sequence $\deg \alpha^{\otimes\beta}$ is strictly increasing in $\beta$.  So
here, we have an element $\delta$ of $T$ given by
$(\alpha^{\otimes(\beta\oplus\gamma')} \oplus
\alpha^{\otimes(\beta'\oplus\gamma)}) \ominus
\alpha^{\otimes(\beta'\oplus\gamma')}$ for some $\beta'<\beta$ and
$\gamma'<\gamma$ and we want to determine its degree.  Now, in general, if we
have ordinals $\alpha$ and $\beta$, then
$\deg(\alpha\oplus\beta)=\max\{\deg\alpha,\deg\beta\}$, and so it follows that
if $\deg\alpha>\deg\beta$ then $\deg(\alpha\ominus\beta)=\deg\alpha$.  So here
it follows that \[\deg\delta=\max\{\deg \alpha^{\otimes(\beta'\oplus\gamma)},
\deg \alpha^{\otimes(\beta\oplus\gamma')}\}.\]

But this means we can find an element of $S$ with degree at least $\deg \delta$;
and since $\beta$ and $\gamma$ are limit ordinals, we can find an element with
degree even larger than $\deg \delta$, which in particular means that $\delta$
is less than some element of $S$.

Therefore $S$ and $T$ are cofinal and so have the supremum.  This completes the
proof.
\end{proof}

As mentioned above, this then implies Theorem~\ref{maincor}:

\begin{proof}[Proof of Theorem~\ref{maincor}]
We prove the more general version by induction on $\gamma$.  If $\gamma=0$, then
\[ \alpha^{\otimes(\bigoplus_{i<0} \beta_i)} = \alpha^{\otimes0} = 1
= \bigotimes_{i<0} \alpha^{\otimes\beta_i},\]
as needed.

If $\gamma$ is a successor ordinal, say $\gamma=S\delta$, then
\begin{multline*}
\alpha^{\otimes(\bigoplus_{i<S\delta} \beta_i)} =
\alpha^{\otimes((\bigoplus_{i<\delta}\beta_i)\oplus \beta_\delta)} =
\alpha^{\otimes(\bigoplus_{i<\delta}\beta_i)}
	\otimes\alpha^{\otimes\beta_\delta} = \\
\bigotimes_{i<\delta}(\alpha^{\otimes\beta_i})
	\otimes\alpha^{\otimes\beta_\delta} =
\bigotimes_{i<S\delta}\alpha^{\otimes\beta_i},
\end{multline*}
again as needed, where we have applied both Theorem~\ref{mainthm} and the
inductive hypothesis.

Finally, if $\gamma$ is a limit ordinal, so $\gamma=\lim_{\delta<\gamma}
\delta$, then
\[
\alpha^{\otimes(\bigoplus_{i<\gamma} \beta_i)} =
\alpha^{\otimes(\lim_{\delta<\gamma} \bigoplus_{i<\delta} \beta_i)} =
\lim_{\delta<\gamma} \alpha^{\otimes(\bigoplus_{i<\delta} \beta_i)} =
\lim_{\delta<\gamma} \bigotimes_{i<\delta} \alpha^{\otimes \beta_i} =
\bigotimes_{i<\gamma} \alpha^{\otimes \beta_i},\]
where here we have used both the inductive hypothesis and the fact that
$\alpha^{\otimes\beta}$ is continuous in $\beta$ (a fact which follows
immediately from the definition).

The restricted version then follows by letting $\beta_i=\beta$ for all $i$.
\end{proof}

Thus we see that super-Jacobsthal exponentiation admits algebraic laws similar
to those followed by ordinary exponentiation and Jacobsthal exponentiation and
complete Table~\ref{anothertable}.

\subsection{Super-Jacobsthal exponentiation: Discussion}

The theorems above raise some questions, analogous to those discussed in
Section~\ref{disc}.  Specifically:

\begin{quest}
Can Theorem~\ref{mainthm} be proven by giving an order-theoretic
interpretation to both sides?  Can the same be done for Theorem~\ref{maincor}?
\end{quest}

Of course, proving it in this way would require first finding an order-theoretic
interpretation for super-Jacobsthal exponentiation; none is currently known.

Even if one cannot do that, there is still the question of improving on the
proof of Theorem~\ref{mainthm} given here by giving a more unified proof.  The
proof given here requires separating out the case where the base $\alpha$ is
finite and handling that case ``computationally''.  A unified proof, if one
could be found, would be preferable.

\section{Natural exponentiation}
\label{noexpsec}

In this section we discuss the question of ``natural exponentiation'' and show
that there is no such thing, that Table~\ref{thetable} is complete as-is.
Table~\ref{thetable} has several vertical families of operations, defined by
transfinite iteration.  This raises the question: Can we continue further the
diagonal family in Table~\ref{thetable}), the sequence of natural operations,
and get a natural exponentiation?

Let us denote such an operation by $e(\alpha,\beta)$, where $\alpha$ is the base
and $\beta$ is the exponent.  In this section we will show that such an
operation cannot exist, unless one is willing to abandon basic properties it
ought to possess.  Now, one could produce a whole list of conditions that such
an operation might be expected to satisfy.  For instance, one might desire:

\begin{enumerate}
\item $e(\alpha,0)=1$.
\item $e(\alpha,1)=\alpha$.
\item $e(0,\alpha)=0$ for $\alpha>0$.
\item $e(1,\alpha)=1$.
\item For $\alpha>1$, $e(\alpha,\beta)$ is strictly increasing in $\beta$.
\item For $\beta>0$, $e(\alpha,\beta)$ is strictly increasing in $\alpha$.
\item $e(\alpha,\beta\oplus\gamma)=e(\alpha,\beta)\otimes e(\alpha,\gamma)$.
\item $e(\alpha,\beta\otimes\gamma)=e(e(\alpha,\beta),\gamma)$.
\item $e(\alpha\otimes\beta,\gamma)=e(\alpha,\gamma)\otimes e(\beta,\gamma)$.
\item $e(2,\alpha)>\alpha$.
\end{enumerate}

But even only a small number of these is enough to cause a contradiction.  In
this section we prove:
\begin{thm}
\label{noexpthm}
There is no natural exponentiation $e(\alpha,\beta)$ on the ordinals
satisfying the following conditions:
\begin{enumerate}
\item $e(\alpha,1)=\alpha$.
\item For $\alpha>0$, $e(\alpha,\beta)$ is weakly increasing in $\beta$.
\item $e(\alpha,\beta)$ is weakly increasing in $\alpha$.
\item $e(\alpha,\beta\oplus\gamma)=e(\alpha,\beta)\otimes e(\alpha,\gamma)$.
\item $e(\alpha,\beta\otimes\gamma)=e(e(\alpha,\beta),\gamma)$.
\end{enumerate}
The same holds if hypothesis (5) is replaced with the following hypothesis (5'):
$e(\alpha\otimes\beta,\gamma)=e(\alpha,\gamma)\otimes e(\alpha,\gamma)$.
\end{thm}

\begin{rem}
The version of this theorem where hypothesis (5') is used was also proven
independently, in slightly stronger form, by Asper\'o and Tsaprounis
\cite{longreals}, using essentially the same means.
\end{rem}

Before we go on and prove this, let us make a note about one way that one could
attempt to define $e(\alpha,\beta)$, even though we know it will not work.
Since addition and multiplication in the surreals agree with natural addition
and natural multiplication on the ordinals, one might attempt to define a
natural exponentiation based on the theory of the surreal exponential
(developed by Gonshor \cite[pp.  143--190]{gonshor}).  One could define
$e(\alpha,\beta)=\exp(\beta \log \alpha)$ for $\alpha>0$, and then define
$e(0,0)=1$ and $e(0,\beta)=0$ for $\beta>0$.  And indeed, the operation on the
surreals defined this way will satisfy all of the desiderata in the long list
above, so long as all terms involved are defined.  But there is one fatal
problem: the ordinals are not closed under this operation.  For instance,
it turns out that, using the usual notation for
surreal numbers, one has \[ \exp(\omega\log\omega) =
\omega^{\omega^{1+1/\omega}},\] which is not an ordinal.  One could attempt to
remedy this by rounding up to the next ordinal, but unsurprisingly the resulting
operation is lacking in algebraic laws.

Now, the proof:

\begin{proof}[Proof of Theorem~\ref{noexpthm}]
Suppose we had such an operation $e(\alpha,\beta)$.  Note that hypotheses (1)
and (4) together mean that if $k$ is finite and positive, then
$e(\alpha,k)=\alpha^{\otimes k}$, and in particular, if $n$ is also finite, then
$e(n,k)=n^k$.  By hypothesis (2), this means that for $n\ge 2$ we have
$e(n,\omega)\ge \omega$.  Let us define $\delta=\deg \deg e(2,\omega)$; since
$e(2,\omega)$ is infinite, this is well-defined.

Observe also that by hypothesis (5), we have for $n$ and $k$ as above,
\[ e(n^k, \alpha) = e(e(n,k),\alpha) = e(n,k\otimes\alpha) =
e(n,\alpha\otimes k) = e(e(n,\alpha),k) = e(n,\alpha)^{\otimes k}.\]
(If we had used instead the alternate hypothesis (5'), this too would prove that
$e(n^k,\alpha)=e(n,\alpha)^{\otimes k}$.)

Given any finite $n\ge 2$, choose some $k$ such that $n\le 2^k$; then by the
above and hypothesis (3),
\[ e(2,\omega) \le e(n,\omega) \le e(2,\omega)^{\otimes k} \]
and so
\[ \deg e(2,\omega) \le \deg e(n,\omega) \le (\deg e(2,\omega)) \otimes k \]
and so
\[ \deg \deg e(2,\omega) \le \deg \deg e(n,\omega) \le \deg \deg e(2,\omega), \]
i.e., $\deg \deg e(n,\omega)=\delta$.

Thus we may define a function $f:\mathbb{N}\to\mathbb{N}$ by defining $f(n)$ to
be the coefficient of $\omega^\delta$ in the Cantor normal form of $\deg
e(n,\omega)$.  Then since $e(n^k,\omega)=e(n,\omega)^{\otimes k}$, we have
$f(n^k)=kf(n)$.  And by the above and hypothesis (3) we have that $f$ is weakly
increasing, since $\deg e(n,\omega)$ is weakly increasing and no term of size
$\omega^{S\delta}$ or higher ever appears in any $\deg e(n,\omega)$.  Finally,
we have that $f(2)\ge 1$.

But no such function can exist; given natural numbers $n$ and $m$, it follows
from the above that
\[ \lfloor \log_m n \rfloor f(m) \le f(n) \le \lceil \log_m n \rceil f(m) \]
or in other words that
\[ \left\lfloor \frac{\log n}{\log m} \right\rfloor \le \frac{f(n)}{f(m)}
\le \left\lceil \frac{\log n}{\log m} \right\rceil. \]

If one takes the above and substitutes in $n^k$ for $n$, one obtains
\[ \left\lfloor k\frac{\log n}{\log m} \right\rfloor \le k\frac{f(n)}{f(m)}
\le \left\lceil k\frac{\log n}{\log m} \right\rceil. \]

But in particular, this means that
\[ k\frac{\log n}{\log m} - 1 \le k\frac{f(n)}{f(m)}
\le k\frac{\log n}{\log m} + 1, \]
or in other words, that
\[ \frac{\log n}{\log m} - \frac{1}{k} \le \frac{f(n)}{f(m)}
\le \frac{\log n}{\log m} + \frac{1}{k}; \]
since this holds for any choice of $k$, we conclude that
\[ \frac{f(n)}{f(m)} = \frac{\log n}{\log m}. \]
But the right hand side may be chosen to be irrational, for instance if $m=2$
and $n=3$; thus, the function $f$ cannot exist, and thus neither can our natural
exponentiation $e$.
\end{proof}

\begin{rem}
Note that the only use of hypotheses (1) and (4) was to show that for $k$ a
positive integer, $e(\alpha,k)=\alpha^{\otimes k}$, so strictly speaking the the
theorem could be stated with (1) and (4) replaced by this single hypothesis.
\end{rem}

\section{Comparison between the operations}
\label{compar}

In Table~\ref{thetable} it was asserted that each operation appearing in the
table is pointwise less-than-or-equal-to those appearing to the right of it in
the table.  In this section we justify that assertion.  Let us state this
formally:

\begin{prop}
\label{ineq}
For any ordinals $\alpha$ and $\beta$, one has:
\begin{enumerate}
\item $\alpha+\beta \le \alpha\oplus\beta$.
\item $\alpha\beta \le \alpha\times\beta \le \alpha\otimes\beta$.
\item $\alpha^\beta \le \alpha^{\times\beta} \le \alpha^{\otimes\beta}$.
\end{enumerate}
\end{prop}

The inequalities $\alpha+\beta\le\alpha\oplus\beta$ and
$\alpha\beta\le\alpha\otimes\beta$ are well known; the inequalities
$\alpha\beta\le\alpha\times\beta$ and $\alpha^\beta\le\alpha^{\times\beta}$ are
due to Jacobsthal \cite{jacobsthal}.  We will give proofs of all of the above
nonetheless.

\begin{proof}
First we prove that $\alpha+\beta\le \alpha\oplus\beta$, by induction on
$\beta$.  If $\beta=0$, both sums are equal to $\alpha$.  If $\beta=S\gamma$,
then by the inductive hypothesis,
\[ \alpha+\beta=S(\alpha+\gamma)\le S(\alpha\oplus\gamma)=\alpha\oplus\beta.\]
Finally, if $\beta$ is a limit ordinal, then since $\alpha\oplus\beta$ is
increasing in $\beta$, we have that
\[ \alpha\oplus\beta \ge \sup_{\gamma<\beta} (\alpha\oplus\gamma)
\ge \sup_{\gamma<\beta} (\alpha+\gamma) = \alpha+\beta. \]
So $\alpha+\beta\le\alpha\oplus\beta$.  It then immediately follows from
transfinite induction and the definitions of each that
$\alpha\beta\le\alpha\times\beta$, and $\alpha^\beta\le\alpha^{\otimes\beta}$.

Next we prove that $\alpha\times\beta\le\alpha\otimes\beta$, again by induction
on $\beta$.  If $\beta=0$, both products are equal to $0$.  If $\beta=S\gamma$,
then by the inductive hypothesis,
\[ \alpha\times\beta=(\alpha\times\gamma)\oplus\alpha\le
(\alpha\otimes\gamma)\oplus\alpha = \alpha\otimes\beta.\]
Finally, if $\beta$ is a limit ordinal, then since $\alpha\otimes\beta$ is
(possibly weakly) increasing in $\beta$, we have that
\[ \alpha\otimes\beta \ge \sup_{\gamma<\beta} (\alpha\otimes\gamma)
\ge \sup_{\gamma<\beta} (\alpha\times\gamma) = \alpha\times\beta. \]
So $\alpha\times\beta\le\alpha\otimes\beta$.  It then immediately follows from
transfinite induction and the definitions of each that
$\alpha^{\times\beta}\le\alpha^{\otimes\beta}$.  This completes the proof.
\end{proof}

Of course, this is not the only possible proof.  For instance, all the above
inequalities could also be proven by comparing Cantor normal forms.  Perhaps
more meaningfully, the inequalities $\alpha+\beta\le\alpha\oplus\beta$ and
$\alpha\beta\le\alpha\otimes\beta$ also both follow immediately from the
order-theoretic interpretation of these operations.  This leaves the question of
order-theoretic proofs of the other inequalities.  Lipparini's order-theoretic
interpretation \cite{lipparini2} of $\alpha\times\beta$ does immediately make it
clear that $\alpha\beta\le\alpha\times\beta$ -- indeed, it shows more generally
that $\sum_i \alpha_i \le \bigoplus_i \alpha_i$.  However, it does not seem to
immediately prove that $\alpha\times\beta\le\alpha\otimes\beta$, so finding an
order-theoretic proof there remains a problem.

\begin{quest}
Can the inequality $\alpha\times\beta\le\alpha\otimes\beta$, and part (3) of
Proposition~\ref{ineq}, be proven by giving order-theoeretic interpretations to
all the quantities involved?  What about the infinitary analogue of part (2)?
\end{quest}

All these inequalities hold equally well, of course, for the infinitary versions
of these operations.  Also, note that if we had a natural exponentiation
$e(\alpha,\beta)$, the same the same style of argument used above to prove
$\alpha+\beta\le\alpha\oplus\beta$ and $\alpha\times\beta\le \alpha\otimes\beta$
could also be used to prove $\alpha^{\otimes\beta}\le e(\alpha,\beta)$, in
accordance with Table~\ref{thetable}.  But, as we showed in the previous
section, there is no natural expoentiation.  However, if one is willing to look
a little bit outside of the ordinals, this line of reasoning could be used to
prove that $\alpha^{\otimes\beta}$ is pointwise at most the surreal exponential
discussed in Section~\ref{noexpsec}.

\subsection*{Acknowledgements} The author is grateful to Juan Arias de Reyna for
suggesting the reference \cite{bachmann}, to P.~Lipparini for suggesting the
reference \cite{1999}, and to J.~C.~Lagarias for help with editing.  The author
is grateful to an anonymous reviewer for suggesting substantial improvements to
the paper.  Work of the author was supported by NSF grants DMS-0943832 and
DMS-1101373.

\end{document}